\documentclass[draft,11pt]{article}

\usepackage{amsfonts,amsmath,amsthm,amscd,amssymb,latexsym,cite,verbatim,texdraw,floatflt,caption2,pb-diagram}

\usepackage[mathscr]{eucal}

\newtheorem{theorem}{Theorem}[section]

\newtheorem{remark}{Remark}[section]
\newtheorem{corollary}{Corollary}[section]
\theoremstyle{definition}

\newcommand{\keywords}{\textbf{Keywords: }\medskip}
\newcommand{\subjclass}{\textbf{Mathematics Subject Classification (2010):}\medskip}
\renewcommand{\abstract}{\textbf{Abstract.}\medskip}

\numberwithin{equation}{section}

\begin{document}


\title{\thanks{This work was supported in part  by the Kyrgyz-Turkish Manas
University (Bishkek / Kyrgyz Republic),  project No.~KTM\"{U}-BAP-2019.FBE.02  and the Volkswagen Foundation (VolkswagenStiftung), program ``From Modeling and Analysis to Approximation''.} Widths of  functional classes defined by majorants of generalized moduli of smoothness
in the spaces  ${\mathcal S}^{p}$ }

\author{Fahreddin Abdullayev, Anatolii Serdyuk and Andrii Shidlich}



\date{}

\maketitle

\begin{abstract}

Exact Jackson-type inequalities are obtained  in terms of best approximations and averaged values of generalized moduli
of smoothness in the spaces ${\mathcal S}^p$. The  values of 
Kolmogorov, Bernstein, linear, and projective widths in the spaces ${\mathcal S}^p$ are found for classes of periodic functions defined by certain conditions on the averaged values of the generalized moduli of smoothness.

\end{abstract}
 
\keywords{Kolmogorov width \and Bernstein width \and best approximation \and module of smoothness \and Jackson-type inequality}

\subjclass{\, 41A17	\and 42A32}



\section{Introduction}

Let $ {\mathcal S}^p$,  $ 1\le p<\infty$, (see, for example,  \cite{Stepanets_2001}, \cite[Ch. 11]{Stepanets_M2005}) be the space of   $ 2\pi$-periodic complex-valued Lebesgue summable functions $f$, defined on the real axis $(f\in L)$,
with   finite norm
 \begin{equation}\label{norm_Sp}
 \|f\|_{_{\scriptstyle  {\mathcal S}^p}} :=
 \Big(\sum _{k \in {\mathbb Z}}|\widehat f(k)|^p\Big)^{1/p},
\end{equation}
where
 $\widehat f(k)=\int_{0}^{2\pi }f(x){\mathrm e}^{-{\mathrm i}k x}\frac {{\mathrm d}x}{2\pi} $ are the Fourier coefficients of the
 function $f$.

In the case $p=2$, the spaces ${\mathcal S}^2$ are ordinary Lebesgue spaces $L_2$ of functions
$f\in L$ with finite norm
 \[
 \|f\|_{_{\scriptstyle  L_2}}=\|f\|_{_{\scriptstyle  {\mathcal S}^2}}=\bigg(\frac 1{2\pi}\int\limits_{0}^{2\pi}|f(t)|^2
 {\mathrm d}t\bigg)^{1/2}.
 \]
For arbitrary $ 1\le p<\infty$, these spaces possess some important properties of Hilbert spaces, in particular, the minimal property of Fourier sums, which will be formulated below in the relation (\ref{Best_app_all}).

An active study of the approximative characteristics of the spaces $ {\mathcal S}^p$ originates from the papers
of Stepanets \cite{Stepanets_2001}, \cite[Ch. 11]{Stepanets_M2005}, \cite{Stepanets_2006}, etc.
Stepanets and Serdyuk \cite{Stepanets_Serdyuk_2002} introduced the notion of $k$th modulus of smoothness in ${\mathcal S}^p$
and proved direct and inverse theorems on approximation in terms of these moduli of smoothness  and
the best approximations of  functions.  Also this topic was investigated actively in
\cite{Sterlin_1972}, \cite{Voicexivskij_2002}, \cite{Serdyuk_2003}, \cite{Vakarchuk_2004},
\cite[Ch.~11]{Stepanets_M2005},  \cite{Vakarchuk_Shchitov_2006},  \cite[Ch.~3]{M_Timan_M2009},
\cite{Savchuk_Shidlich_2011}, \cite{Abdullayev_Ozkartepe_Savchuk_Shidlich_2019},  etc.
In the paper,  this research continues. In particular, exact Jackson-type inequalities are obtained  in terms of best approximations and averaged values of generalized moduli of smoothness in the spaces ${\mathcal S}^p$. The   values of Kolmogorov, Bernstein, linear, and projection widths in the spaces   ${\mathcal S}^p$  are found for classes of periodic functions defined by certain conditions on the averaged values of generalized moduli of smoothness.



\section{Preliminaries}

\subsection{Generalized moduli of smoothness and their averaged values}

Let  $\Phi$ be the set of all  continuous bounded non-negative pair functions
$\varphi(t)$  such that  $\varphi(0)=0$ and the Lebesgue measure of the set
$\{t\in {\mathbb R}:\,\varphi(t)=0\}$ is equal to zero.

Developing ideas of the papers \cite{Shapiro_1968}, \cite{Boman_Shapiro_1971}, \cite{Boman_1980},
for a fixed $\varphi\in \Phi$  define the generalized modulus of smoothness
of the function  $f\in {\mathcal S}^p $  by the equality
\begin{equation}\label{general_modulus}
    \omega_\varphi(f,t)_{_{\scriptstyle  {\mathcal S}^p}}
  :=\sup\limits_{|h|\le t} \Big(\sum_{k\in {\mathbb Z}}
\varphi^p(kh) | \widehat{f}(k) |^p\Big)^{1/p},\quad  t\ge 0.
\end{equation}
Let $\omega_\alpha(f,t)_{_{\scriptstyle  {\mathcal S}^p}} $
be the ordinary modulus of smoothness of $f \in {\mathcal S}^p$ of order
$\alpha>0$,  that is,
\begin{equation}\label{usual_modulus}
    \omega_\alpha(f,t)_{_{\scriptstyle  {\mathcal S}^p}} :=
    \sup\limits_{|h|\le t}\|\Delta_h^\alpha f\|_{_{\scriptstyle  {\mathcal S}^p}} =
     \sup\limits_{|h|\le t} \Big\|\sum\limits_{j=0}^\infty (-1)^j {\alpha \choose j} f(\cdot-jh)
     \Big\|_{_{\scriptstyle  {\mathcal S}^p}} ,
\end{equation}
where
${\alpha \choose j}=\frac {\alpha(\alpha-1)\cdot\ldots\cdot(\alpha-j+1)}{j!}$
for
$j \in \mathbb{N}$ and
${\alpha \choose j}=1$ for
$j=0$.

Since for any ${k}\in {\mathbb Z}$, the Fourier coefficients
 \[
|\widehat { \Delta_h^\alpha f }(k)|=
|1-\mathrm{e}^{-\mathrm{i}kh}|^\alpha |\widehat{f}(k)|
=2^\frac \alpha 2 (1-\cos{kh})^\frac \alpha2 |\widehat{f}(k)|,
 \] 
then in view of (\ref{norm_Sp}) and (\ref{general_modulus}), we have
 \[
 \omega_\alpha(f,t)_{_{\scriptstyle  {\mathcal S}^p}} =\sup\limits_{|h|\le t}\Big(\sum_{k\in {\mathbb Z}}
2^\frac {\alpha p} 2 (1-\cos{kh})^\frac {\alpha p}2 |\widehat{f}(k)|^p\Big)^{1/p}=
 \omega_{\varphi_\alpha}(f,\delta)_{_{\scriptstyle  {\mathcal S}^p}},
 \]
where $\varphi_\alpha(t)=2^\frac \alpha 2 (1-\cos{t})^\frac \alpha2$.
In the general case, such modules were considered, in particular, in  \cite{Vasil'ev_2001}, \cite{Kozko_Rozhdestvenskii_2004},
\cite{Vakarchuk_2016}, \cite{Babenko_Konareva_2019}, etc.


Further, let  $M(\tau )$, $\tau>0$, be the set of all functions $\mu$, bounded
non-decreasing and non-constant on the segment   $[0, \tau]$. By $\Omega _\varphi(f, \tau, \mu, u)_{_{\scriptstyle  {\mathcal S}^p}} $, $u>0$,   denote the average value of the generalized modulus of smoothness
$\omega _\varphi$ of the function  $f$ with the weight $\mu \in M(\tau )$, that is,
 \begin{equation}\label{Mean_Value_Gen_Modulus}
  \Omega _\varphi(f, \tau, \mu, u)_{_{\scriptstyle  {\mathcal S}^p}} :=\bigg (\frac
   {1}{\mu  (\tau ) - \mu  (0)}\int 
   _0^u\omega _\varphi^p(f, t)_{_{\scriptstyle  {\mathcal S}^p}} {\mathrm d}\mu  \Big(\frac {\tau
   t}{u}\Big)\bigg)^{1/p}.
 \end{equation}
In particular,  $\Omega _\alpha(f, \tau, \mu, u)_{_{\scriptstyle  {\mathcal S}^p}} $  denotes
the average value of the modulus of smoothness of the order  $\alpha$ of the function   $f$ with the weight $\mu
\in M(\tau )$, that is,
 $
  \Omega _\alpha(f, \tau, \mu, u)_{_{\scriptstyle  {\mathcal S}^p}} :=\Omega _\varphi(f, \tau, \mu, u)_{_{\scriptstyle  {\mathcal S}^p}}
 $   when $\varphi(t)=\varphi_\alpha(t)=
  2^\frac \alpha 2 (1-\cos{t})^\frac \alpha2.
   $%

Note that for arbitrary $f\in {\mathcal S}^p,$ $\tau >0,$ $\mu
\in M(\tau ),$ $u>0$ the functionals $\Omega _\varphi(f, \tau, \mu, u)_{_{\scriptstyle  {\mathcal S}^p}} $
do not exceed the value $\omega _\varphi(f,u)_{_{\scriptstyle  {\mathcal S}^p}},$ and therefore in a number of questions
they can be more effective for characterizing the structural and approximative properties of the function $f$.


\subsection{Definition of $\psi$-derivatives derivatives and functional classes}

Let  $\psi=\{\psi(k)\}_{k=-\infty}^\infty$ be an arbitrary sequence of complex numbers.
If for a given function $f\in L$ with the Fourier series
 $\sum
 _{k\in {\mathbb Z}}\widehat {f}(k){\mathrm e}^{{\mathrm i}k
 x}
 $
 the series
 $
 \sum
 _{k\in {\mathbb Z}}\psi(k)\widehat {f}(k){\mathrm e}^{{\mathrm i}k
 x}
 $
 is the Fourier of  a certain function  $F\in L$, then   $F$ is called  (see, for example, \cite[Ch.~11]{Stepanets_M2005})
 $\psi$-integral of the function  $f$ and is denoted as $F={\cal J}^{\psi }(f, \cdot)$. In turn,
 the function $f$ is called the $\psi$-derivative of the function $F$ and is denoted as  $f=F^{\psi}$.
 In this case, the Fourier coefficients of functions $f$ and $f^{\psi }$ are related by the equalities
\begin{equation} \label{Fourier_Coeff_der}
 \widehat f(k)=\psi (k)\widehat f^{\psi }(k), \quad k \in {\mathbb Z}.
  \end{equation}
The set of  $\psi$-integrals of
functions  $f$ of $L$ is denoted as  $L^{\psi }$. If
 ${\mathfrak {N}}\subset L$, then  $L^{\psi }{\mathfrak {N}}$ denotes the set of $\psi$-integrals of
 functions $f\in {\mathfrak {N}}$. In particular,    $L^{\psi }{\mathcal S}^p$ is the set
 of $\psi$-integrals of    functions $f\in {\mathcal S}^p$.

In the case when  $\psi (k) = ({\mathrm i}k)^{-r}$, $r=0, 1,\ldots$, we denote  $L^{\psi }=:L^{r}$ and
$L^{\psi }{\mathfrak {N}}=:L^{r}{\mathfrak {N}}$.

For arbitrary fixed $\varphi\in \Phi$, $\tau>0$ and $\mu\in  M(\tau)$, define the following functional classes:
  \begin{equation} \label{L^psi(varphi,n)}
  L^{\psi }(\varphi,\tau ,\mu , n)_{_{\scriptstyle  {\mathcal S}^p}}:=
  \Big\{f\in L^{\psi }{\mathcal S}^p:\quad
  \Omega_\varphi\Big(f^{\psi }, \tau, \mu ,\frac{\tau }n\Big)_{_{\scriptstyle  {\mathcal S}^p}} \le 1, \quad n \in
  {\mathbb{N}}\Big\},
  \end{equation}
   \begin{equation} \label{L^psi(varphi,Phi)}
   L^{\psi }(\varphi, \tau, \mu , \Omega )_{_{\scriptstyle  {\mathcal S}^p}}  :=
    \Big\{f\in L^{\psi }{\mathcal S}^p:\
    \Omega _\varphi(f^{\psi }, \tau , \mu ,  u)_{_{\scriptstyle  {\mathcal S}^p}} \le \Omega  (u),\ 0\le u\le \tau \Big\},
  \end{equation}
where $\Omega  (u)$ is a fixed continuous monotonically increasing function of the variable   $u\ge 0$ such that $\Omega  (0)=0$. Also we set  $L^{\psi }(\alpha,\tau ,\mu , n)_{_{\scriptstyle  {\mathcal S}^p}}:=L^{\psi }(\varphi,\tau ,\mu , n)_{_{\scriptstyle  {\mathcal S}^p}}$ and  $L^{\psi }(\alpha,\tau ,\mu,  \Omega )_{_{\scriptstyle  {\mathcal S}^p}}:=
L^{\psi }(\varphi, \tau, \mu , \Omega )_{_{\scriptstyle  {\mathcal S}^p}}$ for  $\varphi(t)=\varphi_\alpha(t)=
2^\frac \alpha 2 (1-\cos{kh})^\frac \alpha2$.

Note that for  $p=2$, $\psi (k)=k^{-r}$, $r\in \mathbb{N}$, and the weight function
$\mu(t)=t$, Taikov  \cite{Taikov_1976}, \cite{Taikov_1979} first considered the
functional classes  similar to the classes $L^{\psi }(\alpha,\tau ,\mu , n)_{_{\scriptstyle  {\mathcal S}^p}}$ and
$L^{\psi }(\alpha,\tau ,\mu,  \Omega )_{_{\scriptstyle  {\mathcal S}^p}}$.  He found
the exact values of the widths of such classes  in the spaces  $L_2$   in the case when the majorants $\Omega$
of the averaged values of  the moduli of smoothness satisfied some constraints. Later, the problem of finding the exact values of the widths in the spaces $L_2$ and ${\mathcal S}^p$ of functional classes of this kind generated by some specific weighting functions $\mu$, was studied in \cite{Aynulloyev_1984}, \cite[Ch. 4]{Pinkus_1985}, \cite{Yussef_1988},  \cite{Yussef_1990}, \cite{Shalaev_1991},
\cite{Esmaganbetov_1999}, \cite{Serdyuk_2003}, \cite{Vakarchuk_2004}, \cite{Vakarchuk_2016}, etc.


\subsection{Best approximations and widths of functional classes}


Let   ${\mathscr T}_{2n+1}$, $n=0,1,\ldots$, be the set of trigonometric polynomials
${T}_{n}(x) = \sum_{|k|\le n}  c_{k}\mathrm{e}^{\mathrm{i}kx}$ of the order  $n$, where $c_{ k}$
are arbitrary complex numbers.

For any function $f\in {\mathcal S}^p$ denote by $E_n (f)_{_{\scriptstyle  {\mathcal S}^p}}$ its best approximation by the trigonometric polynomials   ${T}_{n-1}\in {\mathscr T}_{2n-1}$ in the space  ${\mathcal S}^p$, that is,
 \begin{equation}\label{Best_Approx_Deff}
    E_n (f)_{_{\scriptstyle  {\mathcal S}^p}} :=
    \inf\limits_{{T}_{n-1}\in {\mathscr T}_{2n-1} }\|f-{T}_{n-1}\|_{_{\scriptstyle  {\mathcal S}^p}}
    .
 \end{equation}
From  relation (\ref{norm_Sp}),  it follows (see, for example,
 \cite[Ch.~11, relation (11.4)]{Stepanets_M2005}) that for any function
$f \in {\mathcal S}^p$ and all $n=0,1,\ldots$,
           \begin{equation} \label{Best_app_all}
           E_n^p (f)_{_{\scriptstyle  {\mathcal S}^p}} =\|f-{S}_{n-1}({f})\|_{_{\scriptstyle  {\mathcal S}^p}} ^p=\sum _{|k |\ge  n}|\widehat f(k )|^p,
           \end{equation}
where $S_{n-1}(f)=S_{n-1}(f,\cdot)= \sum _{|k|\le n-1}\widehat{f}(k) {\mathrm{e}^{\mathrm{i}k\cdot}}$ is the partial Fourier sum
of the order  $n-1$ of the function $f$.


Further, let  $K$ be a convex centrally symmetric subset of ${\mathcal S}^p$ and
 let ${ B }$ be a unit ball of the space  ${\mathcal S}^p$. Let also  $F_N$ be an arbitrary $N$-dimensional subspace of space ${\mathcal S}^p$, $N\in {\mathbb N}$, and
$\mathscr{L}({\mathcal S}^p, F_N)$ be a set of linear operators from ${\mathcal S}^p$ to $F_N$.
 By    $\mathscr {P}({\mathcal S}^p, F_N)$ denote the subset of projection operators of the set ${\mathscr{L}}({\mathcal S}^p, F_N)$,  that is, the set of the operators  $A$ of linear projection onto the set
  $F_N$ such that $Af = f$ when $f\in F_N$.  The quantities
 \[
 b_N(K, {\mathcal S}^p)=\sup\limits _{F_{N+1}}\sup\{\varepsilon>0: \varepsilon { B }\cap F_{N+1}
 \subset K\},
 \]
 \[
 d_N(K, {\mathcal S}^p)=\inf\limits _{F_N}\sup \limits _{f\in K}
 \inf \limits _{u\in F_N}\|f - u \|_{_{\scriptstyle  {\mathcal S}^p}} ,
 \]
 \[
 \lambda _N(K,{\mathcal S}^p)=
 \inf \limits _{F_N}\inf \limits_{A\in {\mathscr {L}}({\mathcal S}^p, F_N)}\sup \limits _{f\in K}
 \|f - Af\|_{_{\scriptstyle  {\mathcal S}^p}} ,
 \]
 \[
 \pi _N(K, {\mathcal S}^p)=\inf \limits_{F_N}\inf \limits _{A\in {\mathscr {P}}({\mathcal S}^p,F_N)}
 \sup \limits _{f\in K}\|f - Af\|_{_{\scriptstyle  {\mathcal S}^p}} ,
 \]
 are called Bernstein, Kolmogorov, linear, and projection $N$-widths of the set $K$ in the space ${\mathcal S}^p$, respectively.



\section{Main results}

\subsection{Jackson-type inequalities}

In this subsection,  Jackson-type inequalities are obtained  in terms of best approximations and averaged values of generalized moduli of smoothness in the spaces ${\mathcal S}^p$.

 \begin{theorem}
       \label{Th.6.3.0}
       Assume that $f\in L^{\psi }{\mathcal S}^p$, $1\le p<\infty$,  $\varphi\in \Phi$, $\tau>0$,  $\mu\in  M(\tau)$
       and $\{\psi (k)\}_{k\in {\mathbb Z}}$ is a sequence of complex numbers such that  $|\psi(k)|\le K<\infty$.
        Then for any   $n\in {{\mathbb N}}$ the following inequality is true:
 \begin{equation}\label{Jackson_Type_Ineq}
    E_n(f)_{_{\scriptstyle  {\mathcal S}^p}} \le
    \bigg(\frac {\mu (\tau ) - \mu (0)}
    {I_{n,\varphi,p}(\tau ,\mu )}\bigg)^{1/p} \nu (n)\,
    \Omega _\varphi\Big(f^{\psi}, \tau,\mu , \frac{\tau }n\Big)_{_{\scriptstyle  {\mathcal S}^p}},
 \end{equation}
where  $\nu(n):=\nu(n,\psi)=\sup_{|k|\ge n} |\psi(n)|$,
 \begin{equation}\label{I_n,varphi,p}
      I_{n,\varphi,p}(\tau ,\mu ):= 
      \mathop{\inf\limits _{k\ge n}}\limits_{k \in {\mathbb N}} \int\limits _0^{\tau }\varphi^p\Big(\frac {k t}n\Big)
      {\mathrm d} \mu   (t).
 \end{equation}
If, in addition, the function $\varphi$ is non-decreasing on the interval
$[0,\tau]$,  the quantity
 $\nu(n)=\max\{|\psi(n)|,|\psi(-n)|\}$, and the condition
 \begin{equation}\label{I_n,varphi,p_Equiv_n}
      I_{n,\varphi,p}(\tau ,\mu )=\int\limits _0^{\tau }\varphi^p(t) {\mathrm d} \mu   (t),
 \end{equation}
holds, then  inequality $(\ref{Jackson_Type_Ineq})$ can not be improved and therefore,
 \begin{equation}
       \label{Jackson_Type_Exact}
       \mathop {\sup\limits _{f\in L^{\psi}{\mathcal S}^p}}\limits _{f\not ={\rm const }}
       \frac {E_n(f)_{_{\scriptstyle  {\mathcal S}^p}} }
       {\Omega_\varphi (f^{\psi }, \tau, \mu, \frac{\tau }{n} )_{_{\scriptstyle  {\mathcal S}^p}} }=
       \bigg(\frac {\mu(\tau ) - \mu (0)}{\int_0^{\tau }\varphi^p(t) {\mathrm d} \mu   (t)}\bigg)^{1/p}\nu (n).
 \end{equation}

\end{theorem}

\begin{proof} Let $f\in L^{\psi }{\mathcal S}^p$, $1\le p<\infty$. By virtue of  (\ref{Fourier_Coeff_der}) and (\ref{Best_app_all}), we have
 \[
 E_n^p(f)_{_{\scriptstyle  {\mathcal S}^p}} =\sum _{|k|\ge n}|\widehat f(k)|^p\le
 \sum _{|k|\ge n}\Big|\frac {\nu(n)
 }{\psi (k)}\Big|^p|\widehat f(k )|^p
 \]
 \begin{equation}
       \label{A6.91}
       =\nu^{\,p} (n)
       \sum _{|k|\ge n}\Big|\frac {\widehat f(k)}{\psi (k)}\Big|^p=\nu^{\,p} (n)
        E_n^p(f^{\psi })_{_{\scriptstyle  {\mathcal S}^p}} .
   \end{equation}

As shown in  \cite[Proof of Theorem 2]{Abdullayev_Chaichenko_Shidlich_2021}, for any   $g \in {\mathcal S}^p$, $1\le p<\infty$, $\tau >0$, $\varphi\in \Phi$,  $\mu \in M(\tau )$ and   $n\in {\mathbb N}$
 \begin{equation}\label{Jackson_type_OLD}
       E_n^p(g)_{_{\scriptstyle  {\mathcal S}^p}} \le
       \frac 1{I_{n,\varphi,p}(\tau ,\mu )}
       \int \limits _0^{\tau }
       \omega_\varphi^p\Big(g, \frac t{n}\Big)_{_{\scriptstyle  {\mathcal S}^p}} {\mathrm d} \mu  (t).
   \end{equation}
Setting  $g =f^{\psi }$ in (\ref{Jackson_type_OLD}), we get
\[
        E_n^p(f^{\psi})_{_{\scriptstyle  {\mathcal S}^p}} \le
        \frac {\mu (\tau ) - \mu (0)}{I_{n,\varphi,p}(\tau ,\mu )}
        \frac {\int  _0^{\tau }\omega _\varphi^p(f^{\psi }, \frac t{n})_{_{\scriptstyle  {\mathcal S}^p}} {\mathrm d} \mu  (t)}{\mu (\tau ) - \mu (0)}
 \]
  \begin{equation}\label{A6.93}
  \le 
        \frac {\mu (\tau ) - \mu (0)}{I_{n,\varphi,p}(\tau ,\mu )}
        \Omega _\varphi^p\Big(f^{\psi }, \tau, \mu , \frac{\tau }n\Big)_{_{\scriptstyle  {\mathcal S}^p}} .
   \end{equation}
Combining  inequalities (\ref{A6.91}) and (\ref{A6.93}), we obtain   (\ref{Jackson_Type_Ineq}).

Now let the function $\varphi$ be non-decreasing on 
$[0,\tau]$, the condition (\ref{I_n,varphi,p_Equiv_n}) holds and $\nu(n)=\max\{|\psi(n)|,|\psi(-n)|\}$.
Then by virtue of (\ref{Jackson_Type_Ineq}), we have
 \begin{equation}\label{A6.94}
      \mathop {\sup\limits _{f\in L^{\psi}{\mathcal S}^p}}\limits _{f\not ={\rm const }}
       \frac {E_n(f)_{_{\scriptstyle  {\mathcal S}^p}} }
       {\Omega_\varphi (f^{\psi }, \tau, \mu, \frac{\tau }n )_{_{\scriptstyle  {\mathcal S}^p}} }\le
             \bigg(\frac {\mu(\tau ) - \mu (0)}{\int_0^{\tau }\varphi^p(u) {\mathrm d} \mu   (u)}\bigg)^{1/p}\nu(n)
       .
   \end{equation}
 To prove the unimprovability of  inequality (\ref{A6.94}),   consider the function
 \[
  f_n(x)=\gamma  +\varepsilon_{-n}\delta   {\mathrm e}^{-{\mathrm i}nx} +
  \varepsilon_{n} \delta    {\mathrm e}^{{\mathrm i}nx},
  \]
where  $\gamma$ and  $\delta $ are arbitrary complex numbers, and  the quantity $\varepsilon_{k}$, $k\in \{-n,n\}$, is  equal to $1$ when $\nu(n)=|\psi(k)|$ and    $\varepsilon_{k}=0$ when $\nu(n)>|\psi(k)|$.

Since the  function $\varphi(nt)$ is non-decreasing on the interval  $[0, \frac {\tau}n]$,
then by virtue of (\ref{general_modulus}) and (\ref{Fourier_Coeff_der}) we have
 \begin{equation}\label{A6.96}
\omega _\varphi (f_n^\psi, t) =    |\delta  |(\varepsilon_{-n} +\varepsilon_{n})^{1/p}\
 \frac {\varphi
 (nt)}{\nu(n)}.
 \end{equation}
Taking into account (\ref{Mean_Value_Gen_Modulus}),  (\ref{A6.96}) and the equality
 $
    E_n
  (f_n)_{_{\scriptstyle  {\mathcal S}^p}}  =|\delta  |(\varepsilon_{-n} +\varepsilon_{n})^{1/p}
  ,
 $
we see that
  \[
  \mathop {\sup\limits _{f\in L^{\psi }{\mathcal S}^p}}\limits _{f\not ={\rm const}}
  \frac {E_n  (f)_{_{\scriptstyle  {\mathcal S}^p}} }
  {\Omega _\varphi  (f^{\psi }, \tau , \mu , \frac {\tau }n)_{_{\scriptstyle  {\mathcal S}^p}} }
  \ge
  \frac {E_n  (f_n)_{_{\scriptstyle  {\mathcal S}^p}} }
  {\Omega _\varphi  (f_n^{\psi }, \tau, \mu, \frac {\tau }n)_{_{\scriptstyle  {\mathcal S}^p}} }
  \]
  \begin{equation}\label{A6.97}
  = \frac {|\delta  |(\varepsilon_{-n} +\varepsilon_{n})^{1/p} (\mu (\tau )-\mu (0))^{1/p} \nu (n)}
   {
    \Big(\int _0^{\tau /n}|\delta  |^p(\varepsilon_{-n} +\varepsilon_{n})\varphi^p(nt){\mathrm d} \mu  (nt)\Big)^{1/p}
   }=
        \Big(\frac {\mu(\tau ) - \mu (0)}{\int_0^{\tau }\varphi^p(u) {\mathrm d} \mu   (u)}\Big)^{1/p} \nu (n).
 \end{equation}
Relations (\ref{A6.94}) and (\ref{A6.97}) yield (\ref{Jackson_Type_Exact}).

\end{proof}

Combining relations (\ref{A6.91}) and (\ref{Jackson_type_OLD}) with $g=f^{\psi }$,
given that the  modulus $\omega_\varphi(f,t)_{_{\scriptstyle  {\mathcal S}^p}}$  is  non-decreasing for $t\ge 0$, we conclude that the following statement holds:

\begin{corollary} \label{Cor.6.3.0}
Assume that $f\in L^{\psi }{\mathcal S}^p$, $1\le p<\infty$, $\varphi\in \Phi$, $\tau>0$, $\mu\in  M(\tau)$ and  $\{\psi (k)\}_{k\in {\mathbb Z}}$ is a  sequence of complex numbers such that  $|\psi(k)|\le K<\infty$.  Then for any  $n\in {{\mathbb N}}$
 \begin{equation}\label{Jackson_type_OLD_PSI}
       E_n(f)_{_{\scriptstyle  {\mathcal S}^p}} \le \bigg(\frac {\mu (\tau ) - \mu (0)}
    {I_{n,\varphi,p}(\tau ,\mu )}\bigg)^{1/p}\,\nu(n)\, \omega_\varphi\Big(f,\frac{\tau}{n}\Big)_{_{\scriptstyle  {\mathcal S}^p}},
   \end{equation}
where  $\nu(n)=\sup_{|k|\ge n} |\psi(n)|$ and the quantity  $I_{n,\varphi,p}(\tau ,\mu )$ is defened by $(\ref{I_n,varphi,p})$.
\end{corollary}

In the case when  $p=2$, $\mu _1(t) = 1 - \cos t$, $\tau=\pi$ and the function
$\varphi_1(t)=2^\frac 12 (1-\cos{t})^\frac 12$, that is, when  $\omega_\varphi$ is the ordinary  modulus of smoothness of the order 1, the inequality of the form  (\ref{Jackson_type_OLD_PSI}) was obtained by Stepanets \cite[Ch.~8]{Stepanets_M2005}.
As follows from formula (\ref{A6.102})  below, in this case 
 $
 \Big(\frac {\mu (\tau ) - \mu (0)}
    {I_{n,\varphi,p}(\tau ,\mu )}\Big)^{1/p}=2^{-1/2}.
 $

The function $ \varphi_\alpha(t)=2^\frac \alpha 2 (1-\cos{t})^\frac \alpha2$, $\alpha>0$, is non-decreasing on the interval   $[0,\pi]$.   Therefore, in this case the following statement  holds:


\begin{corollary} \label{Th.6.3.1}
       Assume that $f\in L^{\psi }{\mathcal S}^p$, $1\le p<\infty$,   $\tau>0$, $\mu\in  M(\tau)$ and  $\{\psi (k)\}_{k\in {\mathbb Z}}$ is a  sequence of complex numbers such that  $|\psi(k)|\le K<\infty$. Then for any numbers
       $\alpha>0$ and $n\in {{\mathbb N}}$
     $$
    E_n(f)_{_{\scriptstyle  {\mathcal S}^p}} \le
    \bigg(\frac {\mu (\tau ) - \mu (0)}
    {I_{n,\alpha,p}(\tau ,\mu )}\bigg)^{1/p}\,\nu(n)\,
    \Omega _\varphi\Big(f^{\psi}, \tau,\mu , \frac{\tau }n\Big)_{_{\scriptstyle  {\mathcal S}^p}},
    \eqno (\ref{Jackson_Type_Ineq}')
    $$
where  $\nu(n)=\sup_{|k|\ge n} |\psi(n)|$, the quantity $I_{n,\alpha,p}(\tau ,\mu )$ is defined by
  $(\ref{I_n,varphi,p})$ with  $ \varphi(t)=\varphi_\alpha(t)=
2^\frac \alpha 2 (1-\cos{t})^\frac \alpha2
$.

If, in addition,  $\nu(n)=\max\{|\psi(n)|,|\psi(-n)|\}$  and
   $$
   I_{n,\alpha,p}(\tau ,\mu )=
   2^\frac {\alpha p} 2 \int\limits _0^{\tau }(1-\cos t)^\frac {\alpha p}2 {\mathrm d} \mu   (t),
   \eqno(\ref{I_n,varphi,p_Equiv_n}')
   $$
then for $\tau \in (0, \pi] $ inequality $(\ref{Jackson_Type_Ineq}')$ can not be improved and thus,
  $$
       \mathop {\sup\limits _{f\in L^{\psi}{\mathcal S}^p}}\limits _{f\not ={\rm const }}
       \frac {E_n(f)_{_{\scriptstyle  {\mathcal S}^p}} }
       {\Omega_\alpha (f^{\psi }, \tau, \mu, \frac{\tau }{n} )_{_{\scriptstyle  {\mathcal S}^p}} }=
       \bigg(\frac {\mu(\tau ) - \mu (0)}{2^\frac {\alpha p} 2\int_0^{\tau }(1-\cos t)^{\frac {\alpha p}2}
       {\mathrm d} \mu   (t)}\bigg)^{1/p}\,\nu(n)\,
  $$
  $$
       =
       \bigg(\frac {\mu(\tau ) - \mu (0)}{2^{\alpha p}\int_0^{\tau }\sin^{\alpha p} \frac t2
       {\mathrm d} \mu   (t)}\bigg)^{1/p}\,\nu(n)\,
       .
       \eqno(\ref{Jackson_Type_Exact}')
  $$

\end{corollary}


Consider some consequences of this statement for specific weight functions  $\mu _1(t) = 1 - \cos t$ and $\mu _2(t)=t$.

\begin{corollary} \label{Cor.6.3.1}
           Assume that $f\in L^{\psi }{\mathcal S}^p$, $1\le p<\infty$,   and  $\{\psi (k)\}_{k\in {\mathbb Z}}$ is a sequence of complex numbers such that  $|\psi(k)|\le K<\infty$.
           Then for any numbers  $\alpha>0$ and $n\in {{\mathbb N}}$
 \begin{equation}\label{A6.98} 
        E_n(f)_{_{\scriptstyle  {\mathcal S}^p}} \le
               \bigg(\frac {2}
               {I_{n,\alpha,p}
               (\pi , \mu _1)}\bigg)^{1/p}
               \Omega_\alpha\Big(f^{\psi }, \pi, \mu _1, \frac {\pi }n\Big)_{_{\scriptstyle  {\mathcal S}^p}}
               \,\nu(n)\,
               ,
  \end{equation}
where  $\nu(n) =\sup_{|k|\ge n} |\psi(n)|$,
 \begin{equation}\label{A6.99} 
       I_{n,\alpha,p}(\pi ,\mu_1)
       =2^{\frac {\alpha p} 2}
       \mathop{\inf\limits _{k\ge n}}\limits_{k \in {\mathbb N}} \int\limits _0^{\pi}\Big(1 - \cos
      \frac {k t}n\Big)^{\frac {\alpha p }2} \sin t \, {\mathrm d}t
  \end{equation}
If, in addition,   $\nu(n)=\max\{|\psi(n)|,|\psi(-n)|\}$  and the number $\frac {\alpha p }2\in {{\mathbb N}}$,
then inequality $(\ref{A6.98})$ on the set $L^{\psi }{\mathcal S}^p$ can not be improved and
  \begin{equation}\label{A6.100} 
  \mathop {\sup\limits _{f\in L^{\psi }{\mathcal S}^p}}\limits _{f\not = {\rm const}}
  \frac {E_n(f)_{_{\scriptstyle  {\mathcal S}^p}} }{\Omega _\alpha (f^{\psi }, \pi, \mu_1,
\frac {\pi }n)_{_{\scriptstyle  {\mathcal S}^p}} } = \frac {(\frac {\alpha p }2 + 1)^{1/p}}{2^\alpha}
 \,\nu(n)
 .
  \end{equation}

\end{corollary}

\begin{proof}

Indeed, if $\tau=\pi$ and $\mu (t)= 1 - \cos t$ in Corollary  \ref{Th.6.3.1}, then relation  (\ref{A6.98})
follows from inequality  $(\ref{Jackson_Type_Ineq}')$. If  the number $\frac {\alpha p}2$
is positive integer, then use the following formula (see \cite[relation (52)]{Stepanets_Serdyuk_2002}):
  \begin{equation}\label{A6.101}
        \inf \limits_{\theta \ge 1}\int \limits _0^{\pi }(1 - \cos \theta t)^{\lambda}
        \sin t {\mathrm d}t = \frac {2^{\lambda +1}}{\lambda +1},\quad   \lambda \in {{\mathbb N}}.
  \end{equation}
Setting $\lambda = \frac{\alpha p}2$ and  $\theta = \frac {k}n$,  $k=n, n+1, n+2, \ldots, $ we see that  $\theta\ge 1$. Therefore,
  \begin{equation}\label{A6.102}
 \mathop{\inf\limits _{k\ge n}}\limits_{k \in {\mathbb N}} \int\limits _0^{\tau }\Big(1 - \cos
      \frac {k t}n\Big)^{\frac {\alpha p }2} \sin t \, {\mathrm d}t=\int\limits _0^{\tau }\Big(1 - \cos
      t\Big)^{\frac {\alpha p }2} \sin t \, {\mathrm d}t= \frac {2^{\frac {\alpha p }2+1}}{\frac {\alpha p }2+1},
  \end{equation}
and equality  (\ref{A6.100}) follows from relation  $(\ref{Jackson_Type_Exact}')$ of Corollary  \ref{Th.6.3.1}
with $\tau =\pi$ and  $\mu (t)=1 - \cos t$.

\end{proof}

\begin{remark}
 In case $p=2$ and $\psi (k) = (ik)^{-r}$, $r=0, 1,\ldots$,  equality  $(\ref{A6.100})$
  can be given in the form
  $$
         \mathop {\sup\limits _{f\in L^r {\mathcal S}^2}}\limits _{f\not ={\rm const}}
         \frac{E_n(f)_{_{\scriptstyle  {\mathcal S}^2}} }
         {\Omega _\alpha (f^{(r)}, \pi, \mu _1, \frac {\pi }n)_{_{\scriptstyle  {\mathcal S}^2}}}
         =\frac{\sqrt {\alpha+1}}{2^\alpha}n^{-r}, \quad \alpha>0, \ n\in {\mathbb N}.
         \eqno (\ref{A6.100}')
 $$
 For $\alpha=1$ this relation follows from the result of Chernykh  \cite{Chernyx_1967_2}.
 For arbitrary  $\alpha=k\in {{\mathbb N}}$ and $n\in {\mathbb N}$, the exact values of the quantities on the
 left-hand side of   $(\ref{A6.100}')$ were obtained by  Yussef \cite{Yussef_1988}  in a slightly different form.
 \end{remark}

\begin{corollary} \label{Cor.6.3.2.}  Let   $0<\tau \le \frac {3\pi }4$, $\mu _2(t)=t$,
  the numbers   $1\le p<\infty$ and  $\alpha>0$ be such that  $\alpha p\ge 1$.
  Let  also $n\in {{\mathbb N}}$   and   $\psi \in \Psi $ be the sequence such that
   $\nu(n)=\sup_{|k|\ge n} |\psi(n)| =\max\{|\psi(n)|,|\psi(-n)|\}$. Then
    \begin{equation}\label{A6.103}
    \mathop{\mathop {\sup}\limits _{f\in L^{\psi }{\mathcal S}^p}}\limits_{f\not ={\rm const }}
           \frac {E_n(f)_{_{\scriptstyle  {\mathcal S}^p}} }
                 {\Omega _\alpha(f^{\psi },\tau , \mu_2, \frac {\tau }n)_{_{\scriptstyle  {\mathcal S}^p}} }
                 = \bigg(\frac {\tau}{2^{\alpha p}\int _0^{\tau }\sin ^{\alpha p }\frac t{2}{\mathrm d}t}\bigg)^{1/p}
                 \nu(n)
                 .
  \end{equation}

\end{corollary}

\begin{proof} As shown in \cite{Voicexivskij_2002}, for arbitrary numbers
 $\tau \in (0, \frac {3\pi }4]$ and $\gamma\ge 1$
  \[
          \mathop {\inf \limits _{k\ge n}}\limits _{k \in {{\mathbb N}}}
          \int \limits_0^{\tau }\Big|\sin \frac {\nu t}{2n}\Big|^{\gamma} {\mathrm d}t=
          \int \limits_0^{\tau } \sin ^{\gamma } \frac t{2} {\mathrm d}t.
  \]
Therefore, for   $\gamma=\alpha p$  and  $\tau \in (0, \frac {3\pi }4]$, we have
 \[
       I_{n,\alpha,p}(\tau ,\mu_2 )=
                    2^{\frac {\alpha p} 2} \mathop {\inf\limits _{k \ge n}}\limits _{k \in {\mathbb N}}
                    \int\limits _0^{\tau }\Big(1 - \cos\frac {k t}n\Big)^{\frac {\alpha p }2} {\mathrm d}t
                    = 2^{\alpha p} \mathop {\inf\limits _{k \ge n}}\limits _{k \in {\mathbb N}}
                    \int\limits _0^{\tau }\Big|\sin \frac {k t}{2n}\Big|^{\alpha p}   {\mathrm d}t
 \]
 \[ 
                    = 2^{\alpha p} \int\limits _0^{\tau }\Big|\sin \frac {t}{2}\Big|^{\alpha p}   {\mathrm d}t
                    = 2^{\frac {\alpha p} 2} \int\limits _0^{\tau }
                      (1 - \cos t )^{\frac {\alpha p }2} {\mathrm d}t.
 \]
Thus,  relation  (\ref{A6.103}) follows from equality  $(\ref{Jackson_Type_Exact}')$ of Corollary    \ref{Th.6.3.1}
with $\mu (t)=t$ and $\tau \in (0,\frac {3\pi }4]$.

\end{proof}

Note that in the case where $p=2$, $\psi (k)=({\mathrm i}k)^{-r},$ $r\ge 0$ and
 $k=1$ or $r\ge 1/2$ and $k \in {{\mathbb N}}$, equality (\ref{A6.103}) follows from the results of Taikov
 \cite{Taikov_1976}, \cite{Taikov_1979},
 (see also \cite{Ligun_1978}).


\subsection{Widths of the classes $L^{\psi }(\varphi, \mu, \tau, n)_{_{\scriptstyle  {\mathcal S}^p}} $}

In this subsection,  the values of Kolmogorov, Bernstein, linear, and projection widths are found for the
the classes $L^{\psi}(\varphi, \mu, \tau, n)_{_{\scriptstyle  {\mathcal S}^p}} $ in the case when the sequences
$\psi (k)$ satisfy some natural restrictions. To state these results,  denote by $\Psi$ the set of arbitrary
sequences $\{\psi (k)\}_{k\in {\mathbb Z}}$  of complex numbers such that  $|\psi (k)|=|\psi (-k)| \ge |\psi (k+1)|$ for $k\in {\mathbb N}$.

\begin{theorem} \label{Th.6.3.2}
       Assume that  $1\le p<\infty,$ $\psi \in \Psi$, $\tau>0$, the function
       $\varphi\in \Phi$ is non-decreasing on the interval $[0,\tau]$ and  $\mu \in M(\tau )$.
       Then for any  $n\in {\mathbb N}$ and  $N\in \{2n-1,  2n\}$  the following inequalities are true:
\[
   \bigg(\frac{\mu (\tau ) - \mu (0)}
   {\int_0^{\tau }\varphi^p(t) {\mathrm d} \mu   (t)}\bigg)^{1/p}|\psi (n)|
   \le P_{N} (L^{\psi }(\varphi, \tau, \mu, n)_{_{\scriptstyle  {\mathcal S}^p}} , {\mathcal S}^p )
\]
\begin{equation}\label{A6.105}
    \le \bigg(\frac {\mu (\tau ) - \mu(0)}
   {I_{n,\varphi,p}(\tau ,\mu )}\bigg)^{1/p}   |\psi (n)|,
  \end{equation}
where the quantity $I_{n,\varphi,p}(\tau ,\mu )$ is defined by $(\ref{I_n,varphi,p})$, and $P_N$
is any of the widths $b_N$, $d_N$, $\lambda _N$ or $ \pi _N$. If, in addition, condition $(\ref{I_n,varphi,p_Equiv_n})$ holds, then
    \begin{equation}\label{A6.106} 
          P_{N} (L^{\psi }(\varphi, \tau, \mu, n)_{_{\scriptstyle  {\mathcal S}^p}} , {\mathcal S}^p )=
          \bigg(\frac{\mu (\tau ) - \mu (0)}
          {\int_0^{\tau }\varphi^p(t) {\mathrm d} \mu   (t)}\bigg)^{1/p}|\psi (n)|.
      \end{equation}
\end{theorem}

\begin{proof}

Based on Theorem \ref{Th.6.3.0}, taking into account the definition of the set $\Psi$, for an arbitrary function
$f\in L^{\psi }(\varphi, \tau , \mu, n)_{_{\scriptstyle  {\mathcal S}^p}} $, we have
\[
            E_n(f)_{_{\scriptstyle  {\mathcal S}^p}} \le
                \bigg(\frac {\mu (\tau ) - \mu(0)}{I_{n,\varphi,p}(\tau ,\mu )}\bigg)^{1/p}
                 \Omega _\varphi\Big(f^{\psi}, \tau , \mu, \frac {\tau}n\Big)_{_{\scriptstyle  {\mathcal S}^p}} |\psi (n)|
\]
     \begin{equation}\label{A6.107}
       \le  \bigg(\frac {\mu (\tau ) - \mu(0)}{I_{n,\varphi,p}(\tau ,\mu )}\bigg)^{1/p}
                |\psi (n)|.
      \end{equation}
Then, taking into account the definition of the projection width $\pi _{N}$, and  relations
(\ref{Best_app_all}) and  (\ref{A6.107}), we conclude that
      \[
           \pi _{2n-1}(L^{\psi }(\varphi, \tau, \mu, n)_{_{\scriptstyle  {\mathcal S}^p}} , {\mathcal S}^p)
           \le E_n(L^{\psi }(\varphi, \tau, \mu, n)_{_{\scriptstyle  {\mathcal S}^p}} )_{_{\scriptstyle  {\mathcal S}^p}}
      \]
     \begin{equation}\label{A6.109} %
          =\sup\limits _{f\in L^{\psi }(\varphi, \tau,\, \mu,\, n)_{_{\scriptstyle  {\mathcal S}^p}} }
          E_n(f)_{_{\scriptstyle  {\mathcal S}^p}}
          \le \bigg(\frac {\mu (\tau ) - \mu(0)}{I_{n,\varphi,p}(\tau ,\mu )}\bigg)^{1/p}|\psi (n)|.
      \end{equation}
Since the widths  $b_N$, $d_N$, $\lambda _N$ and  $\pi _N$ do not increase with increasing $N$  and
     \begin{equation}\label{A6.110}
         b_N(K, X)\le   d_N(K, X) \le \lambda_N(K, X)\le \pi _N(K, X)
      \end{equation}
(see, for example, \cite[Ch.~4]{Tikhomirov_M1976}), then by virtue of (\ref{A6.109}), we get  the estimate from above in
(\ref{A6.105}).

To obtain the necessary lower estimate, it suffices to show that
     \begin{equation}\label{A6.111} 
              b_{2n}(L^{\psi }(k, \mu, \tau, n)_{_{\scriptstyle  {\mathcal S}^p}} , {\mathcal S}^p)\ge
              \bigg(\frac{\mu (\tau ) - \mu (0)}   {\int_0^{\tau }\varphi^p(u) {\mathrm d} \mu   (u)}\bigg)^{1/p}|\psi (n)|=:R_n.
      \end{equation}
In the $(2n+1)$-dimensional space ${\mathscr T}_{2n+1}$ of trigonometric polynomials of order $n$, consider
ball $B_{2n+1}$, whose radius is equal to the number $R_n$ defined in  (\ref{A6.111}), that is,
  $$
  B_{2n+1}=\Big\{t_n\in {\mathscr T}_{2n+1} 
  : \|t_n\|_{_{\scriptstyle  {\mathcal S}^p}} \le R_n\Big\}.
  $$
and prove the embedding  $B_{2n+1}\subset L^{\psi }(\varphi, \tau, \mu, n)_{_{\scriptstyle  {\mathcal S}^p}} $.

For an arbitrary polynomial $T_n\in B_{2n+1}$, due to (\ref{general_modulus}) and the parity of the function
$\varphi$, we have
 $$
 \omega_\varphi^p(T_n^{\psi }, t)_{_{\scriptstyle  {\mathcal S}^p}}
 =\sup\limits_{0\le v\le t}
 \sum _{|k|\le n} \varphi^p(kv) |\widehat T_n^{\psi }(k)|^p.
 $$
Then, taking into account  relation (\ref{Fourier_Coeff_der}) and the nondecreasing of the function $\varphi$  on $[0,a]$,
for $\tau \in (0,a]$  we get
  \[
  (\mu (\tau ) -\mu (0))  \Omega _\varphi^p \Big (T_n^{\psi }, \tau,\mu, \frac {\tau }n\Big )_{_{\scriptstyle  {\mathcal S}^p}}
   =\int\limits _0^{\tau }\omega _\varphi ^p\Big(T_n^{\psi }, \frac t{n}\Big)_{_{\scriptstyle  {\mathcal S}^p}} {\mathrm d} \mu  (t)
   \]
  \[
    =\!\!\int\limits _0^{\tau } \sup\limits _{0\le v\le \frac t{n}}
    \sum\limits  _{|k|\le n} \varphi^p(kv) |\widehat T_n^{\psi }(k)|^p{\mathrm d} \mu  (t)\!=\!\int\limits _0^{\tau } \sup\limits _{0\le v\le t}
    \sum\limits  _{|k|\le n} \varphi^p\Big(\frac {kv}{n}\Big) \Big| \frac{\widehat T_n(k)}{\psi(k)}\Big|^p {\mathrm d} \mu  (t)
  \]
   \[
   \le \frac 1{|\psi(n)|^{p}} \int\limits _0^{\tau }
    \sum\limits  _{|k|\le n} \varphi^p(t) |\widehat T_n(k)|^p {\mathrm d} \mu  (t)\le \frac {\|T_n\|_{_{\scriptstyle  {\mathcal S}^p}} ^p }
    {|\psi(n)|^{p}}
    \int\limits _0^{\tau } \varphi^p(t) {\mathrm d} \mu  (t).
   \]
Therefore, given the inclusion $T_n\in B_{2n+1}$ it follows that
$\Omega _\varphi   (T_n^{\psi }, \tau,\mu, \frac {\tau }n  )_{_{\scriptstyle  {\mathcal S}^p}}\!\le \!1$. Thus,
$T_n\in L^{\psi }(\varphi, \tau, \mu, n)_{_{\scriptstyle  {\mathcal S}^p}}$, the embedding
$B_{2n+1}\subset L^{\psi }(\varphi,\mu , \tau, n)_{_{\scriptstyle  {\mathcal S}^p}}$ is true.  By
the definition of  Bernstein width, the inequality  (\ref{A6.111}) holds. Thus, relation  (\ref{A6.105}) is proved.
It is easy to see that, under the condition (\ref{I_n,varphi,p_Equiv_n}), the upper and lower bounds for
the quantities $P_N (L^{\psi}(\varphi, \tau, \mu, n)_{_{\scriptstyle  {\mathcal S}^p}} ,
{\mathcal S}^p )$ coincide and, therefore, equalities (\ref{A6.106}) hold.

\end{proof}

In the case where the function $ \varphi(t)=
2^\frac \alpha 2 (1-\cos{t})^\frac \alpha2$, we obtain the following statement:


\begin{corollary} \label{Th.6.3.2}
                 Assume that  $1\le p<\infty$, $\psi \in \Psi,$ $\tau \in (0, \pi],$ $\alpha\in {\mathbb N}$ and
                 $\mu  \in M(\tau )$. Then for any  $n\in {\mathbb N}$ and  $N\in \{2n-1,  2n\}$  the following inequalities are true:
    \[ 
                  \bigg(\frac {\mu(\tau ) - \mu (0)}{2^{\alpha p}\int_0^{\tau }\sin^{\alpha p} \frac t2 {\mathrm d} \mu   (t)}\bigg)^{1/p}|\psi (n)|\le P_N(L^{\psi }(\alpha, \tau, \mu, n)_{_{\scriptstyle  {\mathcal S}^p}} , {\mathcal S}^p)
    \]
    \[
    \le
                   \bigg(\frac {\mu (\tau ) - \mu (0)}
                  {I_{n,\alpha,p}(\tau ,\mu )}\bigg)^{1/p}|\psi (n)|,
    \]
where the quantity $I_{n,\alpha,p}(\tau ,\mu )$ is defined by $(\ref{I_n,varphi,p})$ with
$ \varphi(t)=
2^\frac \alpha 2 (1-\cos{kh})^\frac \alpha2$, and $P_N$
is any of the widths $b_N$, $d_N$, $\lambda _N$ or $ \pi _N$. If, in addition, condition $(\ref{I_n,varphi,p_Equiv_n}')$ holds, then
   \[
               P_{N}(L^{\psi }(\alpha, \tau, \mu, n)_{_{\scriptstyle  {\mathcal S}^p}} , {\mathcal S}^p)=
               \bigg(\frac {\mu(\tau ) - \mu (0)}{2^{\alpha p}\int_0^{\tau }\sin^{\alpha p}
               \frac t2 {\mathrm d} \mu   (t)}\bigg)^{1/p}\!\!|\psi (n)|.
    \]

\end{corollary}

For the weight functions  $\mu _1(t)= 1 - \cos t$ and $\mu _2(t)=t$, Corollary  \ref{Th.6.3.2} yields the following statements:

\begin{corollary} \label{Cor_6.3.3.}
                 Assume that  $1\le p<\infty$, $\psi \in \Psi$, $\alpha\in {\mathbb N}$ and  $\mu_1(t)=1-\cos t$.
                 Then for any  $n\in {\mathbb N}$ and  $N\in \{2n-1,  2n\}$
\[
              \frac {(\frac {\alpha p }2+1)^{1/p}}{2^\alpha }|\psi(n)|\le
              P _N(L^{\psi }(\alpha, \pi, \mu _1, n)_{_{\scriptstyle  {\mathcal S}^p}} , {\mathcal S}^p)\le
              \bigg(\frac {2}{I_{n,\alpha,p}
               (\pi , \mu _1)}\bigg)^{1/p}
               |\psi(n)|,
\]
 where $I_{n,\alpha,p}(\pi ,\mu_1)$  is the quantity of the form $(\ref{A6.99})$, and $P_N$
is any of the widths $b_N$, $d_N$, $\lambda _N$ or $ \pi _N$. If, in addition, the number  $\frac {\alpha p}2\in {{\mathbb N}},$ then
      \[
      P_{N}(L^{\psi }(\alpha, \pi, \mu _1, n)_{_{\scriptstyle  {\mathcal S}^p}} ,{\mathcal S}^p)=
       \frac {(\frac {\alpha p }2+1)^{1/p}}{2^\alpha }|\psi(n)|.
       \]
\end{corollary}

\begin{corollary} \label{Cor_6.3.4.}
                  Assume that  $\psi \in \Psi ,$ $0<\tau \le \frac {3\pi }4$, $\mu _2(t)=t$, the numbers
                   $\alpha>0$ and $1\le p<\infty$  such that $\alpha p\ge 1$. Then for any  $n\in {\mathbb N}$ and  $N\in \{2n-1,  2n\}$
     \[ 
     P_{N}(L^{\psi }(\alpha, \tau, \mu _2,  n)_{_{\scriptstyle  {\mathcal S}^p}},{\mathcal S}^p)=
     \bigg(\frac {\tau}{2^{\alpha p}\int _0^{\tau }\sin ^{\alpha p }\frac t{2}{\mathrm d}t}\bigg)^{1/p}
     |\psi (n)|,
     \] 
 where $P_N$ is any of the widths $b_N$, $d_N$, $\lambda _N$ or $ \pi _N$.

\end{corollary}


\subsection{Widths of the classes   $L^{\psi }(\varphi, \mu, \tau, \Omega )_{_{\scriptstyle  {\mathcal S}^p}} $}

Let us find  the widths of the classes $L^{\psi }(\varphi, \mu, \tau,
\Omega )_{_{\scriptstyle  {\mathcal S}^p}} $   that are defined by a majorant $\Omega$ of the averaged values of  generalized moduli of smoothness.

\begin{theorem}
       \label{Th.6.3.3}
       Let  $1\le p<\infty$, $\psi \in \Psi$, the function
       $\varphi\in \Phi$ be non-decreasing on a certain interval $[0,a]$, $a>0$, and
       $\varphi(a)=\sup\{\varphi(t):\, t\in {\mathbb R}\}$.
       Let also  $\tau\in (0,a]$,  the function  $\mu \in M(\tau )$
       and for all $\xi>0$ and $0<u\le a $, the function  $\Omega$ satisfies the condition
     \begin{equation}\label{A6.113} 
     \Omega  \Big(\frac u{\xi }\Big)\bigg(\int\limits _0^{\xi \tau}\varphi_{*}^p(t) {\mathrm d} \mu
     \Big(\frac t{\xi }\Big)\bigg)^{1/p}
     \le \Omega  (u)\bigg(\int \limits _0^{\tau }\varphi^p(t)   {\mathrm d} \mu  (t)\bigg)^{1/p},
      \end{equation}
      where
     \begin{equation}\label{A6.114} 
      \varphi_{*}(t):=\left \{\begin{matrix} \varphi(t),\quad \hfill & 0\le t\le a, \\
      \varphi(a),\quad \hfill &  t\ge a.\end{matrix}\right.
      \end{equation}
 Then for any  $n\in {\mathbb N}$ and  $N\in \{2n-1,  2n\}$  the following inequalities are true:
     \[
            \bigg(\frac{\mu (\tau ) - \mu (0)}{\int_0^{\tau }\varphi^p(t) {\mathrm d} \mu   (t)}\bigg)^{1/p}
            |\psi (n)|\,\,\Omega  \Big(\frac {\tau }n\Big)\le
            P_N(L^{\psi }(\varphi, \tau, \mu, \Omega)_{_{\scriptstyle  {\mathcal S}^p}} , {\mathcal S}^p)
     \]
     \begin{equation}\label{A6.115} %
            \le
            \bigg(\frac {\mu (\tau ) - \mu(0)} {I_{n,\varphi,p}(\tau ,\mu )}\bigg)^{1/p}  |\psi (n)|
            \,\,\Omega  \Big(\frac {\tau }n\Big),
      \end{equation}
where the quantity $I_{n,\varphi,p}(\tau ,\mu )$ is defined by $(\ref{I_n,varphi,p})$, and $P_N$
is any of the widths $b_N$, $d_N$, $\lambda _N$ or $ \pi _N$. If, in addition,  condition $(\ref{I_n,varphi,p_Equiv_n})$ holds, then
     \begin{equation}\label{A6.116} 
     P_{N}(L^{\psi } (\varphi, \tau,  \mu, \Omega  )_{_{\scriptstyle  {\mathcal S}^p}} , {\mathcal S}^p)=
     \bigg(\frac{\mu (\tau ) - \mu (0)}{\int_0^{\tau }\varphi^p(t) {\mathrm d} \mu   (t)}\bigg)^{1/p}
            |\psi (n)|\,\,\Omega  \Big(\frac {\tau }n\Big).
      \end{equation}

\end{theorem}

\begin{proof}  The proof of the theorem basically repeats the proof of Theorem  \ref{Th.6.3.2}.
Based on  inequality (\ref{Jackson_Type_Ineq}), for an arbitrary function
$f\in L^{\psi }(\varphi, \tau, \mu, \Omega)_{_{\scriptstyle  {\mathcal S}^p}}$
   \begin{equation}\label{A6.117} %
            E_n(f)_{_{\scriptstyle  {\mathcal S}^p}} \le
                \bigg(\frac {\mu (\tau ) - \mu(0)}{I_{n,\varphi,p}(\tau ,\mu )}\bigg)^{1/p}
                |\psi (n)| \Omega _\varphi\Big(\frac {\tau}n\Big),
      \end{equation}%
whence, taking into account the definition of the width $\pi _{N}$ and relation (\ref{Best_app_all}), we obtain
\[        
          \pi _{2n-1}(L^{\psi }(\varphi, \mu, \tau, \Omega  )_{_{\scriptstyle  {\mathcal S}^p}} ,{\mathcal S}^p)=
          E_n(L^{\psi }(\varphi, \mu , \tau , \Omega  )_{_{\scriptstyle  {\mathcal S}^p}} )_{_{\scriptstyle  {\mathcal S}^p}}
\]
   \begin{equation}\label{A6.118} 
           =\sup\limits _{f\in L^{\psi }(\varphi, \mu ,\tau, \Omega)_{_{\scriptstyle  {\mathcal S}^p}} }
           E_n(f)_{_{\scriptstyle  {\mathcal S}^p}}
           \le \bigg(\frac {\mu (\tau ) - \mu(0)}{I_{n,\varphi,p}(\tau ,\mu )}\bigg)^{1/p}
                |\psi (n)| \Omega _\varphi\Big(\frac {\tau}n\Big).
      \end{equation}
To obtain the necessary lower estimate, let us show that
    \begin{equation}\label{A6.119} 
    b_{2n}(L^{\psi }(\varphi, \mu , \tau, \Omega  )_{_{\scriptstyle  {\mathcal S}^p}} ,{\mathcal S}^p)\ge
    \bigg(\frac {\mu (\tau ) - \mu(0)}{I_{n,\varphi,p}(\tau ,\mu )}\bigg)^{1/p}
    |\psi (n)| \Omega _\varphi\Big(\frac {\tau}n\Big)=: R_n^*.
      \end{equation}
For this purpose, in the $(2n+1)$-dimensional space ${\mathscr T}_{2n+1}$ of trigonometric polynomials of order $n$, consider
ball $B_{2n+1}$, whose radius is equal to the number $R_n$ defined in (\ref{A6.119}), that is,
   \[ 
   B_{2n+1}^*=\Big\{T_n\in {\mathscr T}_{2n+1} 
  : \|T_n\|_{_{\scriptstyle  {\mathcal S}^p}} \le R_n^*\Big\}
  \]
and prove the validity of the  embedding  $B_{2n+1}^*\subset L^{\psi }(\varphi, \mu , \tau, \Omega  )_{_{\scriptstyle  {\mathcal S}^p}}$.

Assume that $T_n\in B_{2n+1}^*$. Taking into account the non-decrease of the function
$\varphi$ on $[0,a]$  and relations  (\ref{Fourier_Coeff_der}) and (\ref{A6.114}), we have
  \[
  (\mu (\tau ) -\mu (0))\cdot \Omega _\varphi^p  (T_n^{\psi }, \tau,\mu, u )_{_{\scriptstyle  {\mathcal S}^p}}
   =\int\limits _0^{u}\omega _\varphi ^p (T_n^{\psi }, t)_{_{\scriptstyle  {\mathcal S}^p}}
   {\mathrm d} \mu  \Big(\frac {\tau t}{u}\Big)
   \]
  \[
     =\!\!\int\limits _0^{u} \sup\limits _{0\le v\le t}
    \sum\limits  _{|k|\le n} \varphi^p(kv) |\widehat T_n^{\psi }(k)|^p{\mathrm d} \mu  \Big(\frac {\tau t}{u}\Big)\!\!=\!\int\limits _0^{u} \sup\limits _{0\le v\le t}
    \sum\limits  _{|k|\le n} \varphi^p(kv) \Big| \frac{\widehat T_n(k)}{\psi(k)}\Big|^p {\mathrm d} \mu \Big(\frac {\tau t}{u}\Big)
  \]
   \[
   \le \frac 1{|\psi(n)|^{p}} \int\limits _0^{u}
    \sum\limits  _{|k|\le n} \varphi^p_*(n t) |\widehat T_n(k)|^p {\mathrm d} \mu  \Big(\frac {\tau t}{u}\Big)\le \frac {\|T_n\|_{_{\scriptstyle  {\mathcal S}^p}} ^p }
    {|\psi(n)|^{p}}
    \int\limits _0^{u} \varphi^p_*(n t) {\mathrm d} \mu  \Big(\frac {\tau t}{u}\Big).
   \]
   \[
   \le \frac {\|T_n\|_{_{\scriptstyle  {\mathcal S}^p}} ^p }
    {|\psi(n)|^{p}}
    \int\limits _0^{nu} \varphi^p_*(t) {\mathrm d} \mu  \Big(\frac {\tau t}{nu}\Big).
   \]
From the inclusion of  $T_n\in B_{2n+1}^*$ and  relation (\ref{A6.113}) with $\xi =\frac {nu}{\tau} $, it follows  that
   \[
           \Omega _\varphi  (T_n^{\psi }, \tau,\mu, u )_{_{\scriptstyle  {\mathcal S}^p}}\le
           \bigg( \frac {\int _0^{nu} \varphi^p_*(t) {\mathrm d} \mu  (\frac {\tau t}{n u})}
           {\int_0^{\tau }\varphi^p(t) {\mathrm d} \mu   (t)}\bigg)^{1/p}
           \Omega _\varphi\Big(\frac {\tau}n\Big)\le \Omega _\varphi(u).
    \] 
Therefore, indeed  $B_{2n+1}^*\subset L^{\psi }(\varphi, \tau,\mu , \Omega  )_{_{\scriptstyle  {\mathcal S}^p}} $
and by definition of Bernstein width, relation (\ref{A6.119}) is true. Combining relations
(\ref{A6.110}), (\ref{A6.117}) and (\ref{A6.119}), and taking into account
monotonic non-increase of each of the widths $b_N,$  $d_N,$ $\lambda _N$ and $\pi _N$ on $N$, we get
(\ref{A6.115}).  Under the additional condition (\ref{I_n,varphi,p_Equiv_n}), the upper and lower estimates of the quantities
$ P_N(L^{\psi }(\varphi,  \tau, \mu, \Omega  )_{_{\scriptstyle  {\mathcal S}^p}} , {\mathcal S}^p)$ coinside in relation (\ref{A6.115}) and hence, equalities (\ref{A6.116}) are true.

\end{proof}

In the case $ \varphi(t)=\varphi_\alpha(t)=
2^\frac \alpha 2 (1-\cos{t})^\frac \alpha2
$, the following statement is true:


\begin{corollary} \label{Th.6.3.3}
                  Let $1\le p<\infty$, $\psi \in \Psi$, $\tau \in (0, \pi],$ $\alpha>0$ and  $\mu \in M(\tau )$.
                  Let also for all  $\xi>0$ and $0<u\le \pi $, the function $\Omega$  satisfies the condition
                  $$
                  \Omega  \Big(\frac u{\xi }\Big)
                  \bigg(\int\limits _0^{\xi \tau}(1 - \cos t)_{*}^{\frac {\alpha p }2}{\mathrm d} \mu
                  \Big(\frac {t}{\xi}\Big)\bigg)^{1/p}
                  \le
                  \Omega  (u)\bigg(\int \limits _0^{\tau }(1-\cos t)^{\frac {\alpha p }2}{\mathrm d} \mu  (t)\bigg)^{1/p},
                  \eqno (\ref{A6.113}')
                  $$
where
                  $$
                  (1-\cos t)_{*}:=\left \{\begin{matrix} 1-\cos t,\quad \hfill & 0\le t\le \pi, \\
                  2,\quad \hfill &  t\ge \pi.\end{matrix}\right.
                  \eqno (\ref{A6.114}')
                  $$
 Then for any  $n\in {\mathbb N}$ and  $N\in \{2n-1,  2n\}$  the following inequalities are true:
                  $$
                   \bigg(\frac {\mu(\tau ) - \mu (0)}
                   {2^{\alpha p}\int_0^{\tau }\sin^{\alpha p} \frac t2 {\mathrm d} \mu (t)}\bigg)^{1/p}
                   |\psi (n)| \Omega  \Big(\frac {\tau }n\Big)\le
                   P_N(L^{\psi}(\alpha, \tau, \mu, \Omega)_{_{\scriptstyle  {\mathcal S}^p}} , {\mathcal S}^p)
                   $$
                   $$
                   \le \bigg(\frac {\mu (\tau ) - \mu (0)}
                   {I_{n,\alpha,p}(\tau ,\mu )}\bigg)^{1/p}
                   |\psi (n)| \Omega  \Big(\frac {\tau }n\Big),
                   $$
where the quantity $I_{n,\alpha,p}(\tau ,\mu )$ is defined by $(\ref{I_n,varphi,p})$ with
$ \varphi(t)=2^\frac \alpha 2 (1-\cos{kh})^\frac \alpha2$, and $P_N$
is any of the widths $b_N$, $d_N$, $\lambda _N$ or $ \pi _N$.
If, in addition, condition $(\ref{I_n,varphi,p_Equiv_n}')$ holds, then
                 $$
                 P_{N}(L^{\psi }(\alpha, \tau, \mu,  \Omega  )_{_{\scriptstyle  {\mathcal S}^p}} , {\mathcal S}^p)=
                 \bigg(\frac {\mu(\tau ) - \mu (0)}
                 {2^{\alpha p}\int_0^{\tau }\sin^{\alpha p} \frac t2 {\mathrm d} \mu (t)}\bigg)^{1/p}
                 |\psi (n)| \Omega  \Big(\frac {\tau }n\Big).
                 $$

\end{corollary}

Note that for specific weighted functions $\mu  \in M(\tau )$ and some restrictions on other parameters,
the question of the existence of functions $\Omega$ satisfying  conditions of the form (\ref{A6.113}) and $(\ref{A6.113}')$, investigated in  \cite{Taikov_1976}, \cite{Taikov_1979},
\cite{Aynulloyev_1984}, \cite{Yussef_1990}, etc.

For the weight functions  $\mu _1(t)= 1 - \cos t$ and $\mu _2(t)=t$, Corollary \ref{Th.6.3.3} yields the following statements:

\begin{corollary} \label{Cor.6.3.5}
                  Let $1\le p<\infty$, $\psi \in \Psi$, $\mu _1(t)=1-\cos t$ and for all
                  $\xi >0$ and $0<u\le \pi $, the function  $\Omega  $ satisfies the condition
     \begin{equation}\label{A6.121} 
           \Omega  \Big(\frac u{\xi}\Big)
           \bigg (\frac 1{\xi }\int \limits _0^{\pi \xi}(1 - \cos t)_{*}^{\frac {\alpha p }2}\sin \frac t{\xi }dt\bigg)^{1/p}
          \!\!\! \le \Omega  (u)\bigg
           (\int \limits _0^{\pi}(1-\cos t)^{\frac {\alpha p }2}\sin tdt \bigg )^{1/p}\!\!,
     \end{equation} 
where the fucntion  $(1-\cos t)_{*}$ is given by  $(\ref{A6.114}')$. Then for any $n\in {\mathbb N}$ and  $N\in \{2n-1,  2n\}$
 \[
              \frac {(\frac {\alpha p }2+1)^{1/p}}{2^\alpha }|\psi(n)|
              \Omega \Big(\frac {\tau }n\Big)
              \le P_N(L^{\psi }(\alpha,  \pi, \mu _1, \Omega)_{_{\scriptstyle  {\mathcal S}^p}} , {\mathcal S}^p)\le
              \frac {2^{1/p} |\psi(n)|}
              {I^{1/p}_{n,\alpha,p}
               (\pi , \mu _1)}\Omega \Big(\frac {\tau }n\Big),
  \]
where $I_{n,\alpha,p}(\pi ,\mu_1)$  is the quantity of the form $(\ref{A6.99})$, and $P_N$
is any of the widths $b_N$, $d_N$, $\lambda _N$ or $ \pi _N$. If, in addition,  $\frac {\alpha p }2\in {{\mathbb N}},$, then
      \[ 
            P_{N}(L^{\psi }(\alpha,  \pi, \mu_1, \Omega )_{_{\scriptstyle  {\mathcal S}^p}} ,{\mathcal S}^p)=
            \frac {(\frac {\alpha p }2+1)^{1/p}}{2^\alpha }|\psi(n)|
            \Omega \Big(\frac {\tau }n\Big).
      \]  

\end{corollary}

In the case where  $p=2$, $\psi(k)=({\mathrm i}k)^{-r},$ $r\in {{\mathbb N}},$ and $\alpha=1$, the statement  of Corollary
 \ref{Cor.6.3.5} was obtained by Aynulloyev  \cite{Aynulloyev_1984}. In \cite{Aynulloyev_1984}, the existence
 of functions $\Omega$ satisfying condition (\ref{A6.121})  under the above restrictions on the parameters $p$ and $\alpha$
 was also proved.

\begin{corollary} \label{Cor.6.3.6}
                  Let $1\le p<\infty$, $\psi \in \Psi$, $0<\tau \le \frac {3\pi }4,$ $\mu_2 = t$
                 and for all $\xi >0$ and $0<u\le \pi $, the function  $\Omega  $ satisfies the condition
      \begin{equation}\label{A6.123} 
              \Omega  \Big(\frac {u}{\xi }\Big)
              \bigg (\frac 1{\xi}\int \limits _0^{\xi \tau }(1 - \cos t)_{*}^{\frac {\alpha p }2}dt\bigg )^{1/p}
              \le \Omega  (u)
              \bigg (\int \limits _0^{\tau }(1- \cos t)^{\frac {\alpha p }2}dt \bigg )^{1/p}.
      \end{equation} 
Then for any $n\in {\mathbb N}$ and  $N\in \{2n-1,  2n\}$
 \[
            P_{N}(L^{\psi }(\alpha, \tau, \mu_2, \Omega )_{_{\scriptstyle  {\mathcal S}^p}} ,{\mathcal S}^p)=
            \bigg(\frac {\tau}{2^{\alpha p}\int _0^{\tau }\sin ^{\alpha p }\frac t{2}{\mathrm d}t}\bigg)^{1/p}
            |\psi (n)|\Omega  \Big(\frac {\tau }n\Big),
 \]
where $P_N$ is any of the widths $b_N$, $d_N$, $\lambda _N$ or $ \pi _N$.

\end{corollary}
Note that the statements of Corollaries  \ref{Th.6.3.1} (case $\psi \in \Psi$),  \ref{Th.6.3.2}, \ref{Th.6.3.3} and   was proved by Serdyuk \cite{Serdyuk_2003}.

In the case when   $p=2,$ $\psi (k)=({\mathrm i}k)^{-r}$ and  $r\ge 0,$ $\alpha=1$ or
$r\ge 1/2,$ $\alpha\in {{\mathbb N}},$ the statement  of Corollary  \ref{Cor.6.3.6} follows from results of  the papers \cite{Taikov_1976}, \cite{Taikov_1979} (see also \cite[Ch. 4]{Pinkus_1985}), where the existence
 of functions $\Omega$ satisfying (\ref{A6.123}) with the corresponding restrictions on
 $p,$ $\alpha$ and $r$  was also proved.

The question of establishing Jackson-type inequalities in the spaces ${\mathcal S}^p$, as well as finding exact values of the widths of classes generated by averaged values of moduli of smoothness of a form similar to (\ref{Mean_Value_Gen_Modulus}), was considered in \cite{Vakarchuk_2004}, \cite{Vakarchuk_Shchitov_2006},  \cite{Voicexivskij_2002}, etc.




%
%


\begin{thebibliography}{}
%
%


\bibitem{Abdullayev_Chaichenko_Shidlich_2021}
       { F. Abdullayev, S. Chaichenko, A. Shidlich},
       { Direct and inverse approximation theorems of functions in the Musielak-Orlicz type spaces},
       (submitted, see also arXiv preprint, arXiv: 2004.09807).

\bibitem{Abdullayev_Ozkartepe_Savchuk_Shidlich_2019}
        {F.\,G. Abdullayev, P. \"{O}zkartepe, V.\,V. Savchuk, A.\,L. Shidlich},
        { Exact constants in direct and inverse approximation theorems for functions
        of several variables in the spaces ${\mathcal S}^p$},
        Filomat, {\bf 33} (5), 1471-1484 (2019).  

\bibitem{Aynulloyev_1984}
       { N. Aynulloyev},
       { The value of the widths of some classes differentiable functions in $ L_2 $},
       Reports of AS of TadzhSSR, {\bf 29} (8), 415-418 (1985),

\bibitem{Babenko_Konareva_2019}
       { V.\,F. Babenko, S.\,V. Konareva},
       { Jackson-Stechkin-type inequalities for the approximation of elements of Hilbert spaces},
       Ukrainian Math. J.,  {\bf 70} (9), 1331-1344 (2019).

\bibitem{Boman_Shapiro_1971}
       { J. Boman, H.\,S. Shapiro},
       { Comparison theorems for a generalized modulus of continuity},
       Ark. Mat., {\bf 9}, 91-116 (1971).


\bibitem{Boman_1980}
       { J. Boman},
       {  Equivalence of generalized moduli of continuity},
       Ark. Mat., {\bf 18}, 73-100  (1980).

\bibitem{Chernyx_1967_2}
       {N.\,I. Chernykh},
       { Best approximation of periodic functions by trigonometric polinomials in  $L^2$},
       Math. Notes, {\bf 2} (5), 803-808 (1967).


\bibitem{Esmaganbetov_1999}
       {M.\,G. Esmaganbetov},
       { Widths of classes from $L_2 [0, 2\pi]$ and the minimization of exact constants in Jackson-type inequalities},
       Math. Notes, {\bf 65} (6), 689-693 (1999).

\bibitem{Kozko_Rozhdestvenskii_2004}
       {A.\,I. Kozko, A.\,V. Rozhdestvenskii},
       { On Jackson's inequality for a generalized modulus of continuity in $L_2$},
       Sb. Math., {\bf 195} (8), 1073–1115 (2004).


\bibitem{Ligun_1978}
       {A.\,A. Ligun},
       { Some inequalities between best approximations and moduli of continuity in an $L_2$ space},
       Math. Notes, {\bf 24} (6), 917-921 (1978).

\bibitem{Pinkus_1985}
       {A. Pinkus},
       {$n$-Widths in approximation theory},
       Ergebnisse, Springer-Verlag, (1985).


\bibitem{Savchuk_Shidlich_2011}
       {V.\,V. Savchuk,  A.\,L. Shidlich}
       {  Approximation of functions of several variables by linear methods in the space $S^p$},
        Acta Sci. Math.  {\bf 80} (3-4), 477-489 (2014).


\bibitem{Serdyuk_2003}
       {A.\,S. Serdyuk}
       {  Widths in the space ${S}^p$ of classes of functions that are determined by the moduli of continuity of their $\psi$-derivatives},  Proc. Inst. Math. NAS Ukr.   {\bf 46},  229-248 (2003).

\bibitem{Shalaev_1991}
       {V.\,V. Shalaev},
       { Widths in $L_2$ of classes of differentiable functions, defined by higher-order moduli of continuity},
       Ukrainian Math. J., {\bf 43} (1), 104-107 (1991).

\bibitem{Shapiro_1968}
       {H.\,S. Shapiro},
       { A Tauberian theorem related to approximation theory},
       Acta Math., {\bf 120}, 279-292  (1968).


\bibitem{Stepanets_2001}
       {A.\,I. Stepanets},
       { Approximation characteristics of the spaces ${\mathcal S}^p_\varphi$},
       Ukrainian Math. J.,  {\bf 53} (3), 446-475 (2001).


\bibitem{Stepanets_M2005}
       {A.\,I. Stepanets},
       { Methods of approximation theory},
       VSP, Leiden-Boston (2005).

\bibitem{Stepanets_2006}
       {A.\,I. Stepanets},
       { Problems of approximation theory in linear spaces},
       Ukrainian Math. J., {\bf 58} (1), 54–102 (2001). 

\bibitem{Stepanets_Serdyuk_2002}
       {A.\,I. Stepanets, A.\,S. Serdyuk},
       { Direct and inverse theorems in the theory of the approximation of functions in the space ${\mathcal S}^p$},
       Ukrainian Math. J., {\bf 54} (1), 126-148 (2002).

\bibitem{Sterlin_1972}
       {M.\,D. Sterlin},
       { Exact constants in inverse theorems of approximation theory},
       Dokl. Akad. Nauk SSSR, {\bf 202}, 545-547 (1972).




\bibitem{Taikov_1976}
       {L.\,V. Taikov},
       { Inequalities containing best approximations and the modulus of continuity of functions in $L_2$},
       Math. Notes, {\bf 20} (3), 797-800 (1976).

\bibitem{Taikov_1979}
       {L.\,V. Taikov},
       { Structural and constructive characteristics of functions in  $L_2$},
       Math. Notes, {\bf 25} (2), 113-116 (1979).


\bibitem{Tikhomirov_M1976}%
       {V.\,M. Tikhomirov},
       { Some problems in approximation theory},
       Moscow University, Moscow (1976).

\bibitem{M_Timan_M2009}
       {M.\,F. Timan},
       {  Approximation and properties of periodic functions},
       Nauk. dumka, Kiev (2009).

\bibitem{Vakarchuk_2004}
       {S.\,B. Vakarchuk},
       { Jackson-type inequalities and exact values of widths of classes of functions
       in the spaces ${S}^p$, $1\leq p< \infty$},
       Ukrainian  Math. J., {\bf 56} (5), 718-729 (2004).


\bibitem{Vakarchuk_Shchitov_2006}
       {S.\,B. Vakarchuk, A.\,N. Shchitov},
       { On some extremal problems in the theory of approximation of functions in the spaces
        ${S}^p$, $1\leq p< \infty$},
        Ukrainian  Math. J., {\bf 58} (3), 340-356 (2006).

\bibitem{Vakarchuk_2016}
       {S.\,B. Vakarchuk},
       { Jackson-type inequalities with generalized modulus of continuity and exact values of the n-widths for the classes of $(\psi,\beta)$-differentiable functions in $L_2$. I},
       Ukrainian  Math. J., {\bf 68} (6), 823-848 (2016).

\bibitem{Vasil'ev_2001}
       {S.\,N. Vasil'ev},
       { The Jackson-Stechkin inequality in $L_2[-\pi,\pi]$},
       Proc. Steklov Inst. Math.,  Suppl., {\bf 1},  S243-S253 (2001).

\bibitem{Voicexivskij_2002}
       {V.\,R. Voitsekhivs’kyj},
       { Jackson type inequalities in approximation of functions from the space ${\mathcal S}^p$},
       Proc. Inst. Math. NAS Ukr. {\bf 35},   33-46  (2002).


\bibitem{Yussef_1988}
       {Kh. Yussef},
       { On the best approximations of functions and values of widths of classes of functions in $L_2$},
       Collection of scientific works ``Application of functional analysis to the theory
       of Approximations'', Kalinin, 100-114 (1988).


\bibitem{Yussef_1990}
       {Kh. Yussef},
       { Widths of classes of functions in $L_2$},
       Collection of scientific works ``Application of functional analysis to the theory
       of Approximations'', Kalinin,  167-175 (1990).



\end{thebibliography}



AUTHORS

\medskip
\noindent Fahreddin Abdullayev\\
                   Faculty of Sciences, Kyrgyz-Turkish Manas University,\\
                   56, Chyngyz Aitmatov avenue, Bishkek, Kyrgyz republic, 720044;\\
                   Faculty of Science and Letters, Mersin University,\\
                   \c{C}iftlikk\"{o}y Kamp\"{u}s\"{u}, Yeni\c{s}ehir, Mersin, Turkey, 33342 \\
E-mail: fahreddin.abdullayev@manas.edu.kg, fahreddinabdullayev@gmail.com\\

\medskip
\noindent Anatolii Serdyuk \\
Department of Theory of Functions\\
Institute of Mathematics of \\
the National Academy of Sciences of Ukraine\\
Tereschenkivska st., 3, 01024 Kyiv, Ukraine \\
E-mail: sanatolii@ukr.net, serdyuk@imath.kiev.ua\\

\medskip
\noindent Andrii Shidlich \\
Department of Theory of Functions\\
Institute of Mathematics of \\
the National Academy of Sciences of Ukraine\\
Tereschenkivska st., 3, 01024 Kyiv, Ukraine \\
E-mail: shidlich@gmail.com, shidlich@imath.kiev.ua

\end{document}